\documentclass[12pt]{amsart}

\usepackage{amssymb,amsmath}
\usepackage{amssymb,latexsym}
\usepackage{amsfonts}
\usepackage{amssymb}
\usepackage{stmaryrd}
\usepackage{longtable}
\usepackage{hyperref}
\usepackage{amsmath, geometry, amssymb} 
\usepackage{amsaddr}
\numberwithin{equation}{section}
\numberwithin{table}{section}
\newtheorem{theorem}{Theorem}[section]

\newcommand{\smod}[1]{\hspace{-1mm} \pmod{#1}}

\newcommand{\qu}[2]{\Bigl({\frac{#1}{#2}}\Bigr) }
\newcommand{\dqu}[2]{\ds{\qu{#1}{#2}}}

\def \ds{\displaystyle}

\newcommand{\nn}{\mathbb{N}}
\newcommand{\zz}{\mathbb{Z}}
\newcommand{\qq}{\mathbb{Q}}
\newcommand{\cc}{\mathbb{C}}

\newcommand{\hh}{\mathbb{H}}

\allowdisplaybreaks

\begin{document}

\markboth{Zafer Selcuk Aygin}
{Ramanujan-Mordell Formula}
\title{Extensions of Ramanujan-Mordell Formula with Coefficients $1$ and $p$}
\author{Zafer Selcuk Aygin}
\address{Division of Mathematical Sciences, School of Physical and Mathematical Sciences, Nanyang Technological University, 21 Nanyang Link, Singapore 637371, Singapore} 
\email{selcukaygin@ntu.edu.sg}

\begin{abstract}
We use properties of modular forms to prove the following extension of the Ramanujan-Mordell formula, 
\begin{align*}
z^{k-j}z_p^{j}=&\frac{p_{\chi}^{k-j}-1}{p_{\chi}^{k}-1}F_p(k,j;\tau)+ \frac{p_{\chi}^{k}-p_{\chi}^{k-j}}{p_{\chi}^{k}-1}F_p(k,j;p\tau)+z^{k} A_p(k,j;\tau), 
\end{align*}
for all $ 1 < k \in \nn $, $0 \leq j \leq k$  and $p$ an odd prime. We obtain this result by computing the Fourier series expansions of modular forms at all cusps of $\Gamma_0(4p)$. \\


\end{abstract}

\maketitle


\section{Introduction}\label{sec:1}
Let $\nn$, $\nn_0$, $\zz$, $\qq$, $\cc$ and $\hh$ denote the sets of positive integers, non-negative integers, integers, rational numbers, complex numbers and the upper half plane, respectively. Throughout the paper we let $\tau \in \hh$ and  $q=e^{2 \pi i \tau}$. Let $N\in\nn$. Let $\Gamma_0(N)$ be  the modular subgroup defined by
\begin{align*}
\Gamma_0(N) = \left\{ \begin{bmatrix}
    a       & b  \\
    c       & d 
\end{bmatrix}  \mid  a,b,c,d\in \zz ,~ ad-bc = 1,~c \equiv 0 \smod {N}
\right\} .
\end{align*} 

An element {\renewcommand*{\arraystretch}{0.5} $M= \begin{bmatrix}
    a       & b  \\
    c       & d 
\end{bmatrix} \in \Gamma_0(1)$} acts on $\hh \cup \qq \cup \{ \infty \}$ by
\begin{align*}
\ds M(\tau)=\left\{\begin{array}{ll}
	\frac{a\tau+b}{c\tau+d} & \mbox{ if } \tau \neq \infty, \\
	\frac{a}{c} & \mbox{ if } \tau= \infty .
\end{array} \right.
\end{align*}
Let $k \in \nn$. We write $M_k(\Gamma_0(N),\chi)$ to denote the space of modular forms for  $\Gamma_0(N)$ of weight $k$
with multiplier system $\chi$, and $E_k (\Gamma_0(N),\chi)$ and $S_k(\Gamma_0(N),\chi)$ to denote the subspaces of Eisenstein forms and cusp forms of  $M_k(\Gamma_0(N),\chi)$, respectively. When $\chi$ is the primitive principal character we tend to drop the character from the notation.
It is known that $M_k(\Gamma_0(N),\chi)$ is a linear vector space and that, (see for example \cite[p. 83]{stein} and \cite{Serre})
\begin{align}
M_k (\Gamma_0(N),\chi) = E_k (\Gamma_0(N),\chi) \oplus S_k(\Gamma_0(N),\chi). \label{1_1}
\end{align}

Let $\chi$ and $\psi$ be primitive characters. For $n\in \nn$ we define $\ds \sigma_{(k,\chi,\psi )}(n)$ by
\begin{align}
\sigma_{(k, \chi,\psi )}(n) =\sum_{1 \leq d\mid n}\chi(d)\psi(n/d)d^k. \label{3_1}
\end{align}
If $n \not\in \nn$ we set $\sigma_{(k,\chi,\psi )}(n)=0$. For each quadratic discriminant $t$, we put $\chi_{_t}(n)=\dqu{t}{n}$, where $\dqu{t}{n}$ is the Kronecker symbol defined by \cite[p. 296]{vaughan}. Note that, we use $\sigma_{k}(n)$ to denote $\sigma_{(k,\chi_1,\chi_1 )}(n)$, which coincides with the regular sum of divisors function.

Suppose $k \in \nn$ and $p$ an odd prime. The Eisenstein series defined by
\begin{align}
&\displaystyle E_{2k} (\tau)  =1-\frac{4k}{B_{2k}}\sum_{n=1}^{\infty} \sigma_{2k-1}(n)q^{ n }, \nonumber \\
& E^{(1)}_{2k-1}(\tau)= -\frac{4k-2}{B_{2k-1,\chi_{-4}}}\sum_{n \geq 1} \sigma_{(2k-2,\chi_{1}, \chi_{-4} )}(n) q^n, \label{eq2:5} \\
& E^{(2)}_{2k-1}(\tau)=1-\frac{4k-2}{B_{2k-1,\chi_{-4}}}\sum_{n \geq 1} \sigma_{(2k-2,\chi_{-4}, \chi_{1} )}(n) q^n, \label{eq2:6}
\end{align}
will be used to give bases for the spaces $E_{2k}(\Gamma_0(4p))$ and $E_{2k}(\Gamma_0(4p),\chi_{-4})$, see \cite[Theorem 5.9]{stein}. Here Bernoulli numbers $B_{2k}$ and the generalized Bernoulli numbers $B_{2k-1,\chi_{-4}}$ attached to $\chi_{-4}$ are defined by the generating functions
\begin{align*}
& \sum_{k=0}^{\infty} \frac{B_{k}}{k!} x^{k}  =   \displaystyle \frac{x}{e^x-1},\\
& \sum_{k=0}^{\infty} \frac{B_{k,\chi_{-4}}}{k!} x^k = \sum_{a=1}^4\frac{\chi_{-4}(a)x e^{ax}}{e^{4x}-1},
\end{align*}
respectively. For presentation purposes we chose the above normalization for Eisenstein series, which is different from both \cite[(5.3.1)]{stein} and \cite[(7.1.1)]{miyake}. Because of this difference later on we will need the Gauss sum, for a character $\chi$ of conductor $L$, defined by
\begin{align*}
W(\chi)=\sum_{a=0}^{L-1} \chi(a) e^{2 \pi i a \tau/L}.
\end{align*}

Let $m \in \nn$, $r_i \in \nn_0$, and $a_i \in \nn$ for all $1 \leq i \leq m$. Let 
\begin{align*}
N(a_1^{2r_1},a_2^{2r_2},\ldots,a_{m}^{2r_m}; n)
\end{align*}
denote the number of representations of $n$ by the quadratic form
\begin{align}
\sum_{i=1}^{m} \sum_{j=1}^{2r_i} a_{i}x_j^2. \label{eq:1}
\end{align}
Ramanujan's theta function  $\varphi (\tau)$ is defined by
\begin{align*}
\varphi(\tau) = \sum_{n=-\infty}^\infty q^{ n^2 },
\end{align*}
and for $a \in \nn$ we define
\begin{align*}
z_a=\varphi^2(a \tau).
\end{align*}
thus the generating function of number of representations of $n$ by the quadratic form (\ref{eq:1}) is given by
\begin{align*}
\sum_{n=0}^{\infty} N(a_1^{2r_1},a_2^{2r_2},\ldots,a_{m}^{2r_m}; n) q^n= \prod_{i=1}^m \varphi^{2r_i}(a_i \tau)= \prod_{i=1}^m z^{r_i}_{a_i} . 
\end{align*}
Ramanujan in \cite{19ramanujanocaf} gave a formula for $z^{k}$ for $k \in \nn$, from which the value of $N(1^{2k};n)$ follows. This formula was proven by Mordell \cite{mordell}. In 2010, Lemire in his PhD thesis (\cite{lemire}) gave formulas for $N(1^r,2^s,4^t;n)$ for $r \in \nn$, $s, t \in \nn_0$, $r+s+t=4k$. Recently, Cooper et al. in \cite{cooperrmf} gave analogues of Ramanujan-Mordell formula for $(\varphi(\tau) \varphi(p\tau))^k$ for $k \in \nn$ and $p=3,7,11,23$, from which the value of $N(1^{k},p^{k};n)$ follows. In this paper we extend the Ramanujan-Mordell formula with coefficients $1$ and $p$, i.e. we give formulas for $z^{k-j}z_p^{j}$ for $1<k \in \nn$, $0 \leq j \leq k$ and all odd primes $p$. This determines the values of $N(1^{2k-2j},p^{2j};n)$ for all $1<k$, $0 \leq j \leq k$ and $p$ an odd prime. There are some results in the literature which give similar results in terms of products of local densities, see \cite{arenas, survey, siegel}. The strength of our results is that we manage to give contributions from the Eisenstein parts explicitly as opposed to the previous algorithmic results. We use modular forms to prove our results. Our approach is different from the previous applications of modular forms. In the literature, usually the Fourier series expansions at $\infty$ is considered. In this paper, we compute the Fourier series expansions of certain modular forms at all cusps of $\Gamma_0(4p)$, and use them to prove the main theorem. 

The organization of the paper is as follows. In Section \ref{section:main} we state the main theorem. In Section \ref{section:prelim} we introduce the concept of the constant term of modular forms at the cusps $1/c$. We then compute these terms for the Eisenstein series. Then in Section \ref{section:spaces} we take advantage of the fact that the constant terms of Fourier series expansions of cusp forms are always $0$ to obtain some equations and we solve them to give the main terms of any modular form in $M_{2k}(\Gamma_0(4p))$ and $M_{2k-1}(\Gamma_0(4p),\chi_{-4})$. Our particular interest is the extensions of Ramanujan-Mordell formula, which fall into these spaces. In Section \ref{section:proof}, we prove the main theorem which is an application of Theorem \ref{th2:1} with the values of constant terms of the Fourier series expansions of $\varphi^{4k-2j}(\tau) \varphi^{2j}(p\tau)$ and $\varphi^{4k-2j-2}(\tau) \varphi^{2j}(p\tau)$ at cusps of $\Gamma_0(4p)$. The main theorem fails to provide precise description of the cusp part of the formula. In Section \ref{sec:2}, we fix $p=5$ and introduce families of eta quotients which give bases for the spaces $S_{2k}(\Gamma_0(20))$ and $S_{2k-1}(\Gamma_0(20),\chi_{-4})$. We then express the cusp part of the formula as linear combinations of these eta quotients. This basis additionally provides an example of a family of modular form spaces which are generated by eta quotients, a question asked by Ono \cite[Problem 1.68]{onoweb} and recently answered by Rouse and Webb \cite{rouse}.

\section{The main theorem}\label{section:main}

We define
{\small \begin{align*}
 & F_p(2k,j;\tau)= \frac{\chi_{-4}(p)^j}{2^{2k}-1} {  \left( (-1)^k E_{2k}(\tau) -  \left( (-1)^{k}+ \chi_{-4}(p)^{j } \right) E_{2k}(2\tau) + \chi_{-4}(p)^j 2^{2k} E_{2k}(4\tau)  \right)},\\
& F_p(2k-1,j;a \tau)= E^{(2)}_{2k-1}(a\tau) + \chi_{-4}(a) \chi_{-4}(p)^{ j }  ( -4)^{k-1} E^{(1)}_{2k-1}(a\tau), \mbox{ $a \in \{ 1, p \}$.}
\end{align*}}%
Note that the function $F_p(2k,0;\tau)$ is the same function Ramanujan used to give the main terms of his formula. We also define
\begin{align*}
p_{\chi}=\chi_{-4}(p) p=\begin{cases} 
p \mbox{, if $p \equiv 1 \pmod{4}$,}\\
-p \mbox{, if $p \equiv 3 \pmod{4}$.}
\end{cases}
\end{align*}
We are now ready to state our main theorem.
\begin{theorem}
Let $k>1$ be an integer and $0 \leq j \leq k $. Then there exists a modular function $ A_p(k,j;\tau)$ of weight $0$ for $\Gamma_0(4p)$ such that
\begin{align}
z^{k-j}z_p^{j}=&\frac{p_{\chi}^{k-j}-1}{p_{\chi}^{k}-1}F_p(k,j;\tau)+ \frac{p_{\chi}^{k}-p_{\chi}^{k-j}}{p_{\chi}^{k}-1}F_p(k,j;p\tau)+z^{k} A_p(k,j;\tau), \label{eq:16}
\end{align}
and $z^{k} A_p(k,j;\tau)$ is a cusp form.
\end{theorem}
This theorem extends the original Ramanujan-Mordell formula, see  \cite{cooper, mordell, 19ramanujanocaf}. To recover the Ramanujan-Mordell formula we put $j=0$ in (\ref{eq:16}) and obtain
\begin{align*}
z^{k}=&F_p(k,0;\tau)+z^{k} A_p(k,0;\tau),
\end{align*}
where $F_p(k,0;\tau)$ is the same expression with the formulas from  \cite{cooper, mordell, 19ramanujanocaf}. Recently in \cite{cooperrmf}, Cooper et al.\ gave formulas for $(\varphi(\tau)\varphi(p\tau))^k$ valid for all $k \in \nn$ and $p=3,7,11$ and $23$. Letting $1<k \in \nn$ be even, and replacing $j$ by $k/2$ in (\ref{eq:16}), we obtain
\begin{align*}
(\varphi(\tau)\varphi(p\tau))^{k}=&\frac{1}{p_{\chi}^{k/2}+1}F_p(k,k/2;\tau)+ \frac{p_{\chi}^{k/2}}{p_{\chi}^{k/2}+1}F_p(k,k/2;p\tau)+z^{k} A_p(k,k/2;\tau).
\end{align*}
When we put $p=3,7,11$ or $23$, the main terms of the above formula agrees with the main terms of the formulas given by \cite{cooperrmf}. Our formula additionally holds for all odd primes.
\section{ Preliminary results } \label{section:prelim}

In this section we give some theoretic background and then compute the constant terms of Fourier series expansions of Eisenstein series at cusps of $\Gamma_0(4p)$. We use Theorem \ref{th:3} to compute this for $E_{2k}(d\tau)$ for all $d \in \nn$. Then we prove a similar theorem for modular forms in $M_{2k-1}(\Gamma_0(4),\chi_{-4})$. And we finish the computations using a theorem from \cite{miyake}. These results will be used to prove the main theorem. 

A set of representatives of all cusps of $\Gamma_0(4p)$  can be given by
\begin{align*}
R(4p)=\left\{ 1, \frac{1}{2}, \frac{1}{4}, \frac{1}{p}, \frac{1}{2p}, \infty  \right\},
\end{align*}
see \cite[Proposition 2.6]{iwaniec} or \cite{martin}.

Let 
\begin{align}
A_c=\begin{bmatrix}
    -1      & 0  \\
    c       & -1 
\end{bmatrix}, \label{eq:15}
\end{align}
then the Fourier series expansion of $f(\tau) \in M_k(\Gamma_0(N))$ at the cusp $\ds \frac{1}{c}\in \qq $ is given by the Fourier series expansion of $(c\tau+1)^k f(A_c^{-1}\tau)$ at the cusp $\infty$, see \cite[pg. 35]{Kohler}. Let the Fourier series expansion of $f(\tau)$ at the cusp $\ds \frac{1}{c} $ be given by the infinite sum
\begin{align*}
(c\tau+1)^{-k} f(A_c^{-1}\tau)= \sum_{n\geq 0} a_c(n) e^{2 \pi i (n+ \kappa) \tau/h},
\end{align*}
where $h$ the width of $\Gamma_0(N)$ at the cusp, and $0 \leq \kappa < 1$ is the cusp parameter of $f$ at $1/c$. Then we use the notation $[n]_cf(\tau)$ to denote $a_c(n)$. Noting that $[n]_0=[n]_{N}$, for notational convenience we write $[n]_{N}$ instead of $[n]_0$. If we say Fourier series expansion (or Fourier coefficients) without specifying the cusp, we mean the cusp $\infty$. And, for modular forms, `constant term of the Fourier expansion of $f(\tau)$ at cusp $\frac{1}{c}$' refers to the term $[0]_cf(\tau)$. We define $v_c(f)=n+ \kappa$, the order of $f(\tau)$ at $\frac{1}{c}$, where $n$ is the smallest integer such that $[n]_cf(\tau) \neq 0$. Here we should note that, on irregular cusps, the order of the modular form may not be an integer. In this paper, the cusps $1/2$ and $1/{2p} \in R(4p)$ are irregular, and we have
\begin{align*}
& v_{2}(f), v_{2p}(f) \in \nn_0, \mbox{when $f \in M_{2k}(\Gamma_0(4p))$,}\\
& v_{2}(f), v_{2p}(f) \in \nn_0 +\frac{1}{2}, \mbox{when $f \in M_{2k-1}(\Gamma_0(4p),\chi_{-4})$,}
\end{align*}
see \cite[Theorem 2.3.5]{ayginthesis} for details. For the latter case, it turns out we don't need to compute constant terms to prove our results. This appears to be connected to the relationship between the number of cusps and the dimension of Eisenstein spaces.

The following theorem is from an unpublished manuscript from the author.
\begin{theorem}{\cite[Theorem 2.1]{ayginsqfreelevels}} \label{th:3}  
Let $k \in \nn$. Let $f(\tau) \in M_{2k}(\Gamma_0(1))$, with the Fourier series expansion given by
\begin{align*}
f(\tau)=\sum_{n \geq 0} a_n q^n.
\end{align*}
Then for $d \in \nn$, the Fourier series expansion of $f_d(\tau)=f(d\tau)$ at cusp $1/c\in \qq$ is given by
\begin{align*}
 f_{d}(A_c^{-1}\tau)=\Big(\frac{g}{d}\Big)^{2k}(c\tau+1)^{2k} f\Big(\frac{g^2}{d}\tau+\frac{y g}{d}\Big)=\Big(\frac{g}{d}\Big)^{2k}(c\tau+1)^{2k}\sum_{n \geq 0} a_n q_c^n,
\end{align*}
where $g=\gcd(d, c)$,  $y$ is some integer, $A_c$ is the matrix given by {\em (\ref{eq:15})} and $\ds q_c=e^{2\pi i \left(\frac{g^2}{d}\tau+\frac{y g}{d}\right)}$.
\end{theorem}
The proof of the theorem, which is similar to the proof of \cite[Proposition 2.1]{Kohler}, follows from some matrix relations and transformation formula for $f(\tau) \in M_{2k}(\Gamma_0(1))$. We use Theorem \ref{th:3} to obtain the following table of the constant coefficients of Fourier series expansions of $E_{2k}(d\tau)$ at cusps of $\Gamma_0(4p)$ for $k>1$.

\begin{longtable}{l c c c c c c }
\caption{The constant terms of the Fourier series expansions of $E_{2k}(d\tau)$ at cusps of $\Gamma_0(4p)$} \label{table2:2}
\endhead
\hline
cusps &  $E_{2k}(\tau)$  &$E_{2k}(2\tau)$ & $E_{2k}(4\tau)$  & $E_{2k}(p\tau)$ & $E_{2k}(2p\tau)$ & $E_{2k}(4p\tau)$ \\
\hline
$1$   &  $ 1 $  & $\left( \frac{1}{2} \right)^{2k} $ & $ \left( \frac{1}{4} \right)^{2k}$  & $ \left( \frac{1}{p} \right)^{2k} $ & $ \left( \frac{1}{2p} \right)^{2k}$ & $\left( \frac{1}{4p} \right)^{2k} $  \\
 $1/2$   &  $1 $  & $ 1 $ & $ \left( \frac{1}{2} \right)^{2k} $  & $ \left( \frac{1}{p} \right)^{2k} $ & $ \left( \frac{1}{p} \right)^{2k}  $ & $ \left( \frac{1}{2p} \right)^{2k}  $  \\
 $1/4$  &  $ 1 $  & $ 1  $ & $ 1  $  & $ \left( \frac{1}{p} \right)^{2k}  $ & $ \left( \frac{1}{p} \right)^{2k}  $ & $ \left( \frac{1}{p} \right)^{2k}  $  \\
$1/p$  &  $ 1 $  & $ \left( \frac{1}{2} \right)^{2k}$ & $ \left( \frac{1}{4} \right)^{2k}$  & $ 1$ & $ \left( \frac{1}{2} \right)^{2k} $ & $ \left( \frac{1}{4} \right)^{2k} $  \\
$1/2p$  &  $1 $  & $1 $ & $\left( \frac{1}{2} \right)^{2k} $   & $1 $ & $ 1 $ & $ \left( \frac{1}{2} \right)^{2k} $  \\
$\infty$  &  $ 1 $   & $ 1$ & $ 1$  & $1 $ & $1 $ & $ 1$  \\
\hline
\end{longtable}

The next theorem is equivalent of the previous theorem for the modular form space $M_{2k-1}(\Gamma_0(4),\chi_{-4})$, i.e., given that $f(\tau) \in M_{2k-1}(\Gamma_0(4),\chi_{-4})$ and we know the Fourier series expansion of $f(\tau)$ at cusps $1$ and $\infty$, we determine the Fourier series expansions of modular forms $f(p\tau) \in M_{2k-1}(\Gamma_0(4p),\chi_{-4})$ at cusps $1$, $1/4$, $1/p$ and $\infty$. The proof of Theorem \ref{th:4} depends on manipulations of $2\times 2$ matrices. 
\begin{theorem} \label{th:4} Let $k \in \nn$ and $f(\tau) \in M_{2k-1}(\Gamma_0(4),\chi_{-4})$, and the Fourier series expansions of $f(\tau)$ at cusps $1$ and $\infty$, be given by
\begin{align*}
& (\tau+1)^{1-2k} f(A_{1}^{-1} \tau)= \sum_{n\geq 0} a_1 (n) e^{2 \pi i n \tau/4},\\
& f( \tau)= \sum_{n\geq 0} a_0 (n) e^{2 \pi i n \tau},
\end{align*}
respectively. Let $p$ be an odd prime. Then the Fourier series expansions of $f(\tau)$ and $f_p(\tau)=f(p\tau)\in M_{2k-1}(\Gamma_0(4p),\chi_{-4})$ at cusps $1$, $1/4$, $1/p$ and $\infty$ are given by
\begin{align}
& (\tau+1)^{1-2k} f(A_{1}^{-1} \tau)= \sum_{n\geq 0} a_1 (n) e^{2 \pi i n \tau/4}, \label{eq2:1}\\
&  (\tau+1)^{1-2k} f_p(A_{1}^{-1} \tau)= \left( \frac{1}{p} \right)^{2k-1} \sum_{n\geq 0} a_1 (n) e^{2 \pi i  \frac{n(\tau -p_{\chi}p+1)}{4p}},\label{eq:12}\\
& (4 \tau +1)^{1-2k} f(A_{4}^{-1} \tau)= \sum_{n\geq 0} a_0 (n) e^{2 \pi i n \tau},\label{eq2:2} \\
& (4\tau+1)^{1-2k} f_p(A_{4}^{-1} \tau)= \chi_{-4}(p) \left( \frac{1}{p} \right)^{2k-1}  \sum_{n\geq 0} a_0 (n) e^{2 \pi i n \frac{4 \tau - p_{\chi}+1}{4p}},\label{eq:13}\\
& (p\tau+1)^{1-2k} f(A_{p}^{-1} \tau)=\chi_{-4}(p) \sum_{n\geq 0} (\chi_{-4}(p))^{n} a_1 (n) e^{2 \pi i n \tau/4}, \label{eq2:3}\\
&(p\tau+1)^{1-2k}  f_p(A_{p}^{-1} \tau)=   \sum_{n\geq 0} a_1 (n) e^{2 \pi i n p \tau/4},\label{eq:14}\\
& f( \tau)=\sum_{n\geq 0} a_0 (n) e^{2 \pi i n \tau},\label{eq2:4}\\
& f_p( \tau)=  \sum_{n\geq 0} a_{0} (n) e^{2 \pi i n p \tau},\label{eq:2}
\end{align}
respectively. 
\end{theorem}
\begin{proof}
We use the following matrix equations to prove (\ref{eq2:1})--(\ref{eq:2}). The idea is to write a $2\times 2$ matrix into $A=M_1 A_c^{-1} M_2 $, where $M_1 \in \Gamma_0(4)$ and $M_2$ is any matrix with bottom-left entry equal to $0$. As the remaining cases are similar, we prove only (\ref{eq:13}) and in Table \ref{table2:1} we give the matrix equations used to prove other expansions.
{\scriptsize
\begin{longtable}{l l  }
\caption{ Matrices } \label{table2:1}\\
\hline 
Equation &  Matrix Decomposition \\ 
\hline\\
\endfirsthead
\multicolumn{2}{c}{\tablename\ \thetable\ -- \textit{Continued from previous page}} \\
\hline 
Equation &  Matrix Decomposition \\ 
\hline\\
\endhead
  \multicolumn{2}{c}{\textit{Continued on next page}} \\
\endfoot
\hline
\endlastfoot
(\ref{eq2:1}) & $\begin{bmatrix} -1 & 0 \\ -1 & -1 \end{bmatrix}=\begin{bmatrix} 1 & 0 \\ 0 & 1 \end{bmatrix}\begin{bmatrix} -1 & 0 \\ -1 & -1 \end{bmatrix}\begin{bmatrix} 1 & 0 \\ 0 & 1 \end{bmatrix}$ \\[3ex]
(\ref{eq:12}) & $\begin{bmatrix} -p & 0 \\ -1 & -1 \end{bmatrix}=\begin{bmatrix}p-p^2+1 & p^2-1 \\ 1-p & p \end{bmatrix} \begin{bmatrix} -1 & 0 \\ -1 & -1 \end{bmatrix} \begin{bmatrix} 1 & 1-p^2  \\ 0 & p \end{bmatrix}$ \mbox{, if $p \equiv 1 \pmod{4}$} \\[3ex]
 & $\begin{bmatrix} -p & 0 \\ -1 & -1 \end{bmatrix}=\begin{bmatrix} -p^2-p-1 & p^2+1 \\ -p-1 & p \end{bmatrix} \begin{bmatrix} -1 & 0 \\ -1 & -1 \end{bmatrix} \begin{bmatrix} -1 & -p^2-1  \\  0 & -p \end{bmatrix}$ \mbox{, if $p \equiv 3 \pmod{4}$} \\[3ex]
(\ref{eq2:2}) & $\begin{bmatrix} -1 & 0 \\ -4 & -1 \end{bmatrix}=\begin{bmatrix} 1 & 0 \\ 0 & 1 \end{bmatrix}\begin{bmatrix} -1 & 0 \\ -4 & -1 \end{bmatrix}\begin{bmatrix} 1 & 0 \\ 0 & 1 \end{bmatrix}$ \\[3ex]
(\ref{eq:13}) & $\begin{bmatrix} -p & 0 \\ -4 & -1 \end{bmatrix} =\begin{bmatrix} p & (p-1)/4  \\ 4 & 1 \end{bmatrix} \begin{bmatrix} 1 & 0 \\ 0 & 1 \end{bmatrix} \begin{bmatrix} -1 & (p-1)/4  \\ 0 & -p \end{bmatrix}$ \mbox{, if $p \equiv 1 \pmod{4}$}\\[3ex]
& $\begin{bmatrix} -p & 0 \\ -4 & -1 \end{bmatrix} =\begin{bmatrix} p & -(p+1)/4  \\ 4 & -1 \end{bmatrix} \begin{bmatrix} 1 & 0 \\ 0 & 1 \end{bmatrix} \begin{bmatrix} -1 & -(p+1)/4  \\ 0 & -p \end{bmatrix}$ \mbox{, if $p \equiv 3 \pmod{4}$}\\[3ex]
(\ref{eq2:3}) & $\begin{bmatrix} -1 & 0 \\ -p & -1 \end{bmatrix}=\begin{bmatrix} 1 & 0 \\ p-1 & 1 \end{bmatrix}\begin{bmatrix} -1 & 0 \\ -1 & -1 \end{bmatrix}\begin{bmatrix} 1 & 0 \\ 0 & 1 \end{bmatrix}$ \mbox{, if $p \equiv 1 \pmod{4}$} \\[3ex]
& $\begin{bmatrix} -1 & 0 \\ -p & -1 \end{bmatrix}=\begin{bmatrix} 1 & -2 \\ p+1 & -2p-1 \end{bmatrix}\begin{bmatrix} -1 & 0 \\ -1 & -1 \end{bmatrix}\begin{bmatrix} -1 & 2 \\ 0 & -1 \end{bmatrix}$ \mbox{, if $p \equiv 3 \pmod{4}$}\\[3ex]
(\ref{eq:14}) & $\begin{bmatrix} -p & 0 \\ -p & -1 \end{bmatrix}=\begin{bmatrix} 1 & 0 \\ 0 & 1 \end{bmatrix}\begin{bmatrix} -1 & 0 \\ -1 & -1 \end{bmatrix}\begin{bmatrix} p & 0 \\ 0 & 1 \end{bmatrix}$ \\[3ex]
 (\ref{eq2:4}) & $\begin{bmatrix} 1 & 0 \\ 0 & 1 \end{bmatrix}=\begin{bmatrix} 1 & 0 \\ 0 & 1 \end{bmatrix}\begin{bmatrix} 1 & 0 \\ 0 & 1 \end{bmatrix}\begin{bmatrix} 1 & 0 \\ 0 & 1 \end{bmatrix}$ \\[3ex]
 (\ref{eq:2}) & $\begin{bmatrix} p & 0 \\ 0 & 1 \end{bmatrix}=\begin{bmatrix} 1 & 0 \\ 0 & 1 \end{bmatrix}\begin{bmatrix} 1 & 0 \\ 0 & 1 \end{bmatrix}\begin{bmatrix} p & 0 \\ 0 & 1 \end{bmatrix}$  \\[3ex]
 \hline
\end{longtable}}
Now we prove (\ref{eq:13}), when $p \equiv 3 \pmod 4$. Using the matrix decomposition given in Table \ref{table2:1} and $f(\tau)$ being in $ M_{2k-1}(\Gamma_0(4),\chi_{-4})$ we have 
 \begin{align*}
 f_p(A_{4}^{-1} \tau)&= f \left( \begin{bmatrix} p & -(p+1)/4  \\  4 & -1 \end{bmatrix}  \begin{bmatrix} -1 & -(p+1)/4  \\  0 & -p \end{bmatrix} (\tau) \right)\\
&= \chi_{-4}(-1) \left(4 \begin{bmatrix} -1 & -(p+1)/4  \\  0 & -p \end{bmatrix} (\tau)  -1\right)^{2k-1}  f\left(   \begin{bmatrix} -1 & -(p+1)/4  \\  0 & -p \end{bmatrix} (\tau) \right)\\
&= \chi_{-4}(p) \left( \frac{4\tau +1}{p}  \right)^{2k-1}  f\left(   \frac{4\tau -p_{\chi}+1}{4p} \right).
 \end{align*}
 \end{proof}
To conclude this section we give the table of Fourier series expansions of $E_{k,\chi_{-4},\chi_{1}}(\tau)$ and $E_{k,\chi_{1},\chi_{-4}}(\tau)$ at regular cusps of $\Gamma_0(4p)$. The Fourier series expansions of $E^{(2)}_{2k-1}(\tau)$ and $E^{(1)}_{2k-1}(\tau)$ at $\infty$ are already known, see (\ref{eq2:5}) and (\ref{eq2:6}). We use \cite[Lemma 7.1.2]{miyake} to obtain the following equalities, which allow us to compute desired Fourier series expansions at the cusp $1$. The Gauss sum $W(\chi)$ in the below formulas appear due to different choice of normalization of Eisenstein series.
\begin{align*}
 E^{(2)}_{2k-1}\left(\frac{-1}{\tau}\right)&= \frac{W(\chi_{1})}{W(\chi_{-4})} (\tau)^{2k-1}   E^{(1)}_{2k-1}\left(\frac{\tau}{4}\right)= \frac{-i}{2} (\tau)^{2k-1}   E^{(1)}_{2k-1}\left(\frac{\tau}{4}\right),\\
 E^{(1)}_{2k-1}\left(\frac{-1}{\tau}\right)&= -\frac{W(\chi_{-4})}{W(\chi_{1})4^{2k-1}} (\tau)^{2k-1}   E^{(2)}_{2k-1}\left(\frac{\tau}{4}\right)= \frac{-2i}{4^{2k-1}} (\tau)^{2k-1}   E^{(2)}_{2k-1}\left(\frac{\tau}{4}\right).
\end{align*}
We use $\begin{bmatrix} -1 & 0 \\ -1 & -1 \end{bmatrix}=\begin{bmatrix} 1 & 1 \\ 0 & 1 \end{bmatrix}\begin{bmatrix} 0 & 1 \\ -1 & 0 \end{bmatrix}\begin{bmatrix} 1 & 1 \\ 0 & 1 \end{bmatrix}$ to obtain
\begin{align*}
 (\tau+1)^{1-2k} E^{(2)}_{2k-1}(A_1^{-1}\tau)&=  \frac{-i}{2}    E^{(1)}_{2k-1}\left(\frac{\tau+1}{4}\right),\\
(\tau+1)^{1-2k} E^{(1)}_{2k-1}(A_1^{-1}\tau)&=  \frac{-2i}{4^{2k-1}}    E^{(2)}_{2k-1}\left(\frac{\tau+1}{4}\right).
\end{align*}

Now we turn to Theorem \ref{th:4} to obtain the following table. 
 
\begin{longtable}{l c c c c}
\caption{The constant terms of the Fourier series expansions of $E^{(2)}_{2k-1}(\tau)$ and $E^{(1)}_{2k-1}(\tau)$} \label{table2:3}
\endhead
\hline
cusps &  $E^{(2)}_{2k-1}(\tau)$  & $E^{(2)}_{2k-1}(p\tau)$  & $E^{(1)}_{2k-1}(\tau)$  & $E^{(1)}_{2k-1}(p\tau)$ \\
\hline
 $1$ &  $ 0 $  & $0 $ & $  \frac{-2i}{4^{2k-1}} $ & $ \frac{-2i}{{(4p)}^{2k-1}} $ \\
$1/4$ &  $1$  & $ \frac{\chi_{-4}(p)}{p^{2k-1}}  $ & $0$  & $0 $  \\
$1/p$  &  $ 0 $   & $0 $ & $\frac{-2 i \chi_{-4}(p) }{4^{2k-1}} $  & $ \frac{-2 i}{4^{2k-1}}  $  \\
 $\infty$ &  $ 1$  & $ 1  $ & $0 $  & $ 0$ \\
\hline
\end{longtable}

\section{ The spaces $M_{2k}(\Gamma_0(4p))$ and $M_{2k-1}(\Gamma_0(4p),\chi_{-4})$  }\label{section:spaces}
In this paper our particular interest is on extensions of the Ramanujan-Mordell formula. However our calculations yield to the following theorem, which provides information on any modular form in $M_{2k}(\Gamma_0(4p))$ and $M_{2k-1}(\Gamma_0(4p),\chi_{-4})$. For notational convenience let us fix
\begin{align*}
& A_k(p,t,f)= \frac{ {[0]_tf}  - \chi_{-4}(p)^{tk}[0]_{tp}f  }{(2^{k}-1) (p_{\chi}^{k}-1)}\\
& B_k(p,t,f)=\frac{p_{\chi}^{k}[0]_tf- \chi_{-4}(p)^{tk}[0]_{tp}f }{(2^{k}-1) (p_{\chi}^{k}-1)}
\end{align*}
for $t \mid 4$.
\begin{theorem}\label{th2:1} Let $k>1$ be an integer and $p$ be an odd prime. If $f(\tau) \in M_{2k}(\Gamma_0(4p))$, then there exists a cusp form $C_{2k,4p}(\tau) \in S_{2k}(\Gamma_0(4p))$ such that
\begin{align*}
f(\tau)= \sum_{d \mid 4p} b_d(p,f) E_{2k}(d \tau) + C_{2k,4p}(\tau)
\end{align*}
where
{
\begin{align*}
&b_1(p,f)= {2^{2k}B_{2k}(p,1,f)}  -  {B_{2k}(p,2,f) } ,\\ 
& b_2(p,f)=  {-2^{2k} B_{2k}(p,1,f) }   +  {\left( 2^{2k}+1 \right) B_{2k}(p,2,f) }- { B_{2k}(p,4,f)  },\\
&b_4(p,f)= -2^{2k} \left( B_{2k}(p,2,f)  - B_{2k}(p,4,f) \right),\\ 
& b_p(p,f)= -p^{2k} \left( {2^{2k} A_{2k}(p,1,f)  -  A_{2k}(p,2,f)} \right),\\ 
& b_{2p}(p,f)= p^{2k} \left( {2^{2k} A_{2k}(p,1,f)} - { \left( 2^{2k}+1 \right) A_{2k}(p,2,f)} + { A_{2k}(p,4,f)}\right) ,\\
&b_{4p}(p,f)=  {(2p)^{2k} \left( A_{2k}(p,2,f) -  A_{2k}(p,4,f) \right) }.
\end{align*}
}
If $f(\tau) \in M_{2k-1}(\Gamma_0(4p),\chi_{-4})$, then there exists a cusp form
$C_{2k-1,4p}(\tau) $ \\
$\in S_{2k-1}(\Gamma_0(4p),\chi_{-4})$ such that
\begin{align*}
f(\tau)= & a_1(p,f) E^{(2)}_{2k-1}(\tau)+ a_2(p,f) E^{(2)}_{2k-1}(p\tau)\\
& + a_3(p,f) E^{(1)}_{2k-1}(\tau) + a_4(p,f) E^{(1)}_{2k-1}(\tau) + C_{2k-1,4p}(\tau)
\end{align*}
where 
\begin{align*}
& a_1(p,f)= (2^{2k-1}-1)B_{2k-1}(p,4,f)      ,      \\
& a_2(p,f)= (2^{2k-1}-1) p_{\chi}^{2k-1} A_{2k-1}(p,4,f)    ,      \\
& a_3(p,f)= \frac{i4^{2k-1}(2^{2k-1}-1)}{2} B_{2k-1}(p,1,f)   ,       \\
& a_4(p,f)= \frac{-i(4p)^{2k-1}(2^{2k-1}-1)}{2} A_{2k-1}(p,1,f)    .       
\end{align*}
\end{theorem}

\begin{proof}
By \cite[Theorem 5.9]{stein} and (\ref{1_1}), for any $f(\tau)\in M_{2k}(\Gamma_0(4p))$ we have
\begin{align*}
f(\tau)=\sum_{d \mid 4p} b_d E_{2k}(d\tau) + C_{2k,4p}(\tau)
\end{align*}
for some $b_d \in \cc$ and $C_{2k,4p}(\tau) \in S_{2k}(\Gamma_0(4p))$. We use the fact that constant coefficients of cusp forms are $0$ at all cusps to obtain
\begin{align}
[0]_c f(\tau)=\sum_{d \mid 4p}  b_d [0]_c E_{2k}(d\tau) + 0, \label{eq2:7}
\end{align}
for all $c \mid 4p$. The entries of the matrix of system of linear equations determined by (\ref{eq2:7}) is given by Table \ref{table2:2}. We solve this system to obtain desired equations for $b_d$. The second part of the theorem can be proven similarly by using the entries given by Table \ref{table2:3}.
\end{proof}
For brevity we don't state Theorem \ref{th2:1} for weight $2$ spaces. One can use (\ref{eq:21}) and above arguments to give the statement for $M_{2}(\Gamma_0(4p))$.
\section{ Proof of the main theorem  } \label{section:proof}
By Jacobi's triple product identity \cite[p. 10]{SprtRmnj} we have
\begin{align*}
z_a=\varphi^2(a\tau)=\left(\frac{\eta^{5}(2a\tau)}{\eta^2(a\tau) \eta^2(4a\tau)}\right)^2. 
\end{align*}
That is, we can rewrite the generating function in terms of eta quotients:
\begin{align*}
&z^{2k-j-1}z_p^{j}= \varphi^{4k-2j-2}(\tau)\varphi^{2j}(p\tau)=  \left( \frac{\eta^5(2 \tau)}{\eta^2( \tau) \eta^2(4 \tau)} \right)^{4k-2j-2} \left( \frac{\eta^5(2p\cdot \tau)}{\eta^2(p\cdot \tau) \eta^2(4p \cdot \tau)} \right)^{2j},\\
&z^{2k-j}z_p^{j}=  \varphi^{4k-2j}(\tau)\varphi^{2j}(p\tau)=  \left( \frac{\eta^5(2 \tau)}{\eta^2( \tau) \eta^2(4 \tau)} \right)^{4k-2j} \left( \frac{\eta^5(2p\cdot \tau)}{\eta^2(p\cdot \tau) \eta^2(4p \cdot \tau)} \right)^{2j}. 
\end{align*}
By Ligozat Theorem (\cite{Ligozat}, or \cite[Theorem 2.1]{ayginsten}), for $k \in \nn$ we have
\begin{align*}
& \varphi^{4k-2j-2}(\tau)\varphi^{2j}(p\tau) \in M_{2k-1}(\Gamma_0(4p),\chi_{-4}), \mbox{ for all $0 \leq j \leq 2k+1$, }\\
& \varphi^{4k-2j}(\tau)\varphi^{2j}(p\tau)\in M_{2k}(\Gamma_0(4p)), \mbox{ for all $0 \leq j \leq 2k$. }
\end{align*}
We compute
\begin{align*}
[0]\varphi^2(\tau)=1 \mbox{, and } [0]_1\varphi^2(\tau)=\frac{-i}{2},
\end{align*}
using \cite[Proposition 2.1]{Kohler}. Then by Theorem \ref{th:4} we compute the following table for $k \in \nn$ (or, alternatively one can use \cite[Proposition 2.1]{Kohler} to compute all the entries without using Theorem \ref{th:4}.)
\begin{longtable}{l c c c c c c }
\caption{ The constant terms of the Fourier series expansions of $\varphi^{2k-2j}(\tau)\varphi^{2j}(p\tau)$} \label{table:1}
\endhead
\hline
cusps &  $1$  & $1/2$ & $1/4$  & $1/p$ & $1/2p$ & $\infty$ \\
\hline
$\varphi^{4k-2j-2}(\tau)\varphi^{2j}(p\tau)$  &  $\ds  \frac{(-1)^{k} i}{2^{2k-1}p^j} $ & NA & $ \ds \frac{\chi_{-4}(p)^j}{p^j}  $ & $\ds \frac{(-1)^{k} \chi_{-4}(p)^{j-1}  i}{2^{2k-1}} $ & NA & $ 1 $  \\
$\varphi^{4k-2j}(\tau)\varphi^{2j}(p\tau)$   & $\ds \frac{(-1)^k }{2^{2k}p^{j}}$ & $ 0 $ & $ \ds  \frac{\chi_{-4}(p)^j}{p^j}  $  & $  \ds  \frac{\left( -1 \right)^{k} \chi_{-4}(p)^{j}}{2^{2k}}$ & $ 0$ & $1 $  \\
\hline
\end{longtable}
Then we put the values of constant terms of generating functions at the cusps given by Table \ref{table:1} in Theorem \ref{th2:1}. Thus, for $k>1$, we have that there exist cusp forms $C_{2k,4p}(\tau) \in S_{2k}(\Gamma_0(4p))$ and $C_{2k-1,4p}(\tau) \in S_{2k-1}(\Gamma_0(4p),\chi_{-4})$ such that
\begin{align*}
\varphi^{4k-2j}(\tau) \varphi^{2j}(p\tau)= & \frac{p_{\chi}^{2k-j}-1}{p_{\chi}^{2k}-1}F_p(2k,j;\tau)+ \frac{p_{\chi}^{2k}-p_{\chi}^{2k-j}}{p_{\chi}^{2k}-1}F_p(2k,j;p\tau) + C_{2k,4p}(\tau),\\
\varphi^{4k-2j-2}(\tau) \varphi^{2j}(p\tau)=&  \frac{p_{\chi}^{2k-1-j}-1}{p_{\chi}^{2k-1}-1}F_p(2k-1,j;\tau)+ \frac{p_{\chi}^{2k-1}-p_{\chi}^{2k-1-j}}{p_{\chi}^{2k-1}-1}F_p(2k-1,j;p\tau)\\
&+ C_{2k-1,4p}(\tau).
\end{align*}
On the other hand $\ds A_p(k,j;\tau)=\frac{C_{k,4p}(\tau)}{z^{k}}$ is a modular function of weight $0$ for $\Gamma_0(4p)$, from which the theorem follows. Note that, the poles of $A_p(k,j;\tau)$ occur at cusps $\frac{1}{2}$ and $\frac{1}{2p}$, a similar feature is present in the original Ramanujan-Mordell formula.

We didn't state Theorem \ref{th2:1} for weight $2k=2$ spaces for brevity. Below we sketch the proof of main theorem for weight $2$ spaces.  By \cite[Theorem 5.9]{stein} the set $\{ L_{d}(\tau) : 1<d \mid 4p \}$ form a basis for $E_{2}(\Gamma_0(4p))$, where $L_{d}(\tau)=E_2(\tau)-d E_2(d \tau)$. We use \cite[(1.21)]{Kohler} to compute
\begin{align}
[0]_c L_d(\tau)=1-\frac{\gcd(c,d)^2}{d}, \mbox{ for all $c,d \mid 4p$, $d>1$. } \label{eq:21}
\end{align}
The result for weight $2$ extension follows from solving the equations
\begin{align*} 
[0]_c z^{2-j} z_p^{j}=[0]_c \varphi^{4-2j}(\tau) \varphi^{2j}(p \tau)= \sum_{1<d \mid 4p} b_d [0]_c L_d(\tau)=\sum_{1<d \mid 4p}  b_d\left(1-\frac{\gcd(c,d)^2}{d} \right),
\end{align*}
($c \mid 4p$) for $b_d$.
\section{Eta quotients generating $S_{k}(\Gamma_0(20),\chi)$}\label{sec:2}
For $p \leq 13$ it is possible to express $z^kA_p(k,j;\tau)$  in terms of eta quotients, see \cite[Corollary 3]{rouse}. The case $p=3$, has been given in author's PhD thesis, see \cite[Theorem 5.1.3]{ayginthesis}. In this section we give the bases for $S_{2k}(\Gamma_0(20))$ and $S_{2k-1}(\Gamma_0(20),\chi_{-4})$ in terms of eta quotients. We then express $A_5(k,j;\tau)$ as linear combinations of eta quotients. Let us define the following eta quotients, which will be used to express basis elements. 
\begin{align*}
& S_1(k,l ;\tau)=\left(\frac{\varphi(5\tau) }{\varphi(\tau) } \right)^{k} \left(\frac{\eta^{3}(2\tau) \eta^{5}(5\tau) \eta^{10}(20\tau) }{\eta(\tau) \eta^{2}(4\tau) \eta^{15}(10\tau)  }\right)^{l} \left(\frac{ \eta^{12 }( 2\tau)  \eta^{12 }( 5\tau) \eta^{ 12}( 20\tau) }{ \eta^{ 4}( \tau) \eta^{ 4}( 4\tau) \eta^{ 28}( 10\tau)  } \right),\\
& S_2(k,l ;\tau)=\left(\frac{\varphi(5\tau) }{\varphi(\tau) } \right)^{k} \left(\frac{\eta^{3}(2\tau) \eta^{5}(5\tau) \eta^{10}(20\tau) }{\eta(\tau) \eta^{2}(4\tau) \eta^{15}(10\tau)  }\right)^{l}\left(\frac{ \eta^{8 }( 2\tau) \eta^{13 }(5 \tau)  \eta^{ 15}(20 \tau)  }{ \eta( \tau) \eta^{ 3}(4 \tau)  \eta^{ 32}( 10\tau) } \right),\\
& S_3(k,l ;\tau)=\left(\frac{\varphi(5\tau) }{\varphi(\tau) } \right)^{k} \left(\frac{\eta^{3}(2\tau) \eta^{5}(5\tau) \eta^{10}(20\tau) }{\eta(\tau) \eta^{2}(4\tau) \eta^{15}(10\tau)  }\right)^{l} \left(\frac{\eta^{ 9}( 2\tau) \eta^{ 15}( 5\tau)  \eta^{ 18}( 20\tau)  }{  \eta^{ 3}( \tau) \eta^{ 2}( 4\tau)\eta^{ 37}( 10\tau) } \right).
\end{align*}
\begin{theorem} {\label{th:5_2}} Let $k \in \nn $. The sets of eta quotients
\begin{align*}
& \left\{  z^2 S_1(6,0;\tau) \right\},\\
& \left\{ z^{4k}S_1(10k,l;\tau) \mid 0\leq l \leq 4k-3 \right\}\cup \left\{ z^{4k}S_2(10k,l;\tau) \mid 0\leq l \leq 4k-3 \right\} \\
&\qquad \cup \left\{ z^{4k}S_3(10k,l;\tau) \mid 0\leq l \leq 4k-3 \right\},\\
& \left\{ z^{4k-2}S_1(10k-4,l;\tau) \mid 0\leq l \leq 4k-4 \right\}\cup \left\{ z^{4k-2}S_2(10k-6,l;\tau) \mid 0\leq l \leq 4k-5 \right\} \\
&\qquad \cup \left\{ z^{4k-2}S_3(10k-6,l;\tau) \mid 0\leq l \leq 4k-6 \right\},
\end{align*}
form a basis for $S_{2}(\Gamma_0(20))$, $S_{4k}(\Gamma_0(20))$ (for $k\geq 1$) and $S_{4k-2}(\Gamma_0(20))$ (for $k > 1$), respectively; and 
\begin{align*}
& \left\{ z^{4k-1}{S_1(10k-2,l;\tau)} \mid 0\leq l \leq 4k-3 \right\}\cup \left\{ z^{4k-1}S_2(10k-2,l;\tau) \mid 0\leq l \leq 4k-4 \right\} \\
&\qquad \cup \left\{ z^{4k-1}S_3(10k-2,l;\tau) \mid 0\leq l \leq 4k-4 \right\},\\
& \left\{ z^{4k-3}{S_1(10k-6,l;\tau)} \mid 0\leq l \leq 4k-5 \right\}\cup \left\{ z^{4k-3}S_2(10k-8,l;\tau) \mid 0\leq l \leq 4k-6 \right\} \\
&\qquad \cup \left\{ z^{4k-3}S_3(10k-8,l;\tau) \mid 0\leq l \leq 4k-6 \right\},
\end{align*}
form a basis for $S_{4k-1}(\Gamma_0(20),\chi_{-4})$ (for $k\geq 1$) and $S_{4k-3}(\Gamma_0(20),\chi_{-4})$ (for $k> 1$), respectively.
\end{theorem}
\begin{proof}
We use Ligozat Theorem (\cite{Ligozat}, or \cite[Theorem 2.1]{ayginsten}) to check each eta quotient is a cusp form in the corresponding space, and use \cite[Proposition 6.1 and pg. 98]{stein} to compute
\begin{align*}
& {\rm dim}(S_{2}(\Gamma_0(20)))=1,\\
& {\rm dim}(S_{4k}(\Gamma_0(20)))=12k-6,\\
& {\rm dim}(S_{4k-2}(\Gamma_0(20)) )=12k-12,\\
& {\rm dim}( S_{4k-1}(\Gamma_0(20),\chi_{-4}))=12k-8,\\
& {\rm dim}( S_{4k-3}(\Gamma_0(20),\chi_{-4}))=12k-14.
\end{align*}
Further, the orders of zeros of eta quotients at $\infty$ in each set are different, from which the linear independence follows. 
\end{proof}
Additionally this choice of bases give a lower triangular shape to the corresponding matrix system, which allows iterative determination of the coefficients of eta quotients in linear combinations to represent $z^kA_5(k,j;\tau)$.
\begin{theorem}
Let $k>1$ be an integer and $0 \leq j \leq k $. Then we have
\begin{align*}
z^{k-j}z_5^{j}=&\frac{5^{k-j}-1}{5^{k}-1}F_5(k,j;\tau)+ \frac{5^{k}-5^{k-j}}{5^{k}-1}F_5(k,j;5\tau)+z^{k} A_5(k,j;\tau),
\end{align*}
where
{\scriptsize \begin{align*}
&A_5(k,j;\tau)=\\
& \begin{cases}
\ds \frac{4}{3} (1-(-1)^j) S_1(6,0;\tau) , \mbox{ if $k=2$,}\\
\ds \sum_{l=0}^{k-3} \alpha_{3l+1} S_1\left(\frac{5k}{2},l;\tau \right) + \sum_{l=0}^{k-3} \alpha_{3l+2} S_2\left(\frac{5k}{2},l;\tau \right) + \sum_{l=0}^{k-3} \alpha_{3l+3}  S_3\left(\frac{5k}{2}k,l;\tau \right), \mbox{ if $4 \mid k$,} \\
\ds \sum_{l=0}^{k-2} \alpha_{3l+1} S_1\left(\frac{5k+3}{2},l;\tau \right) + \sum_{l=0}^{k-3} \alpha_{3l+2} S_2\left(\frac{5k-1}{2},l;\tau \right) + \sum_{l=0}^{k-3} \alpha_{3l+3}  S_3\left(\frac{5k-1}{2},l;\tau \right), \mbox{ if $4 \mid k-1$,} \\
\ds \sum_{l=0}^{k-2} \alpha_{3l+1} S_1\left(\frac{5k+2}{2},l;\tau \right) + \sum_{l=0}^{k-3} \alpha_{3l+2} S_2\left(\frac{5k-2}{2},l;\tau \right) + \sum_{l=0}^{k-4} \alpha_{3l+3}  S_3\left(\frac{5k-2}{2},l;\tau \right), \mbox{ if $4 \mid k-2$,} \\
\ds \sum_{l=0}^{k-2} \alpha_{3l+1} S_1\left(\frac{5k+1}{2},l;\tau \right) + \sum_{l=0}^{k-3} \alpha_{3l+2} S_2\left(\frac{5k+1}{2},l;\tau \right) + \sum_{l=0}^{k-3} \alpha_{3l+3}  S_3\left(\frac{5k+1}{2},l;\tau \right), \mbox{ if $4 \mid k-3$,} 
\end{cases}
\end{align*} }
and
{\small \begin{align*}
\alpha_l= [l]z^{k-j}z_5^{j} - \frac{5^{k-j}-1}{5^{k}-1}[l]F_5(k,j;\tau) - \frac{5^{k}-5^{k-j}}{5^{k}-1}[l]F_5(k,j;5\tau) - ( [l]z^{k} A_5(k,j;\tau) - \alpha_l) .
\end{align*}}
\end{theorem}
The iteration given for $\alpha_l$ makes sense and computationally efficient, since $[l]S_a(*,n;\tau)=0$ for all $3n+a>l$, and $[3l+a]S_a(*,l;\tau)=1$. For an execution of a similar iteration see \cite[(7.1.9)]{ayginthesis}.

\section*{Acknowledgments} 
The author would like to thank the anonymous referee for his comments which was helpful in improving the exposition. The author was supported by the Singapore Ministry of Education Academic Research Fund, Tier 2, project number MOE2014-T2-1-051, ARC40/14.

\noindent
Zafer Selcuk Aygin\\
Division of Mathematical Sciences \\
School of Physical and Mathematical Sciences \\
Nanyang Technological University \\
21 Nanyang Link, Singapore 637371, Singapore

\vspace{1mm}

\noindent
selcukaygin@ntu.edu.sg

\end{document}